\documentclass[12pt,a4paper,svgnames]{amsart}
\usepackage[english]{babel}
\usepackage{amsthm,amsmath,amsfonts,amssymb}
\usepackage{hyperref}
\usepackage{mathtools}
\usepackage{amsmath,amsfonts}
\usepackage{graphicx}
\usepackage{tikz}
 \usepackage[pagewise]{lineno}
\newcommand{\R}{\ensuremath{\mathbb{R}}}
\newcommand{\N}{\ensuremath{\mathbb{N}}}

\def \ds {\displaystyle}  
\def \tn {\textnormal}  
\def \bac {\backslash}  
\newtheorem{theorem}{Theorem}[section]
\newtheorem{lemma}[theorem]{Lemma}

\newtheorem{definition}[theorem]{Definition}
\newtheorem{proposition}[theorem]{Proposition}
\newtheorem{remark}{Remark}[section]

\makeatletter
\newcommand{\tpitchfork}{%
  \vbox{
    \baselineskip\z@skip
    \lineskip-.52ex
   \lineskiplimit\maxdimen
    \m@th
    \ialign{##\crcr\hidewidth\smash{$-$}\hidewidth\crcr$\pitchfork$\crcr}
  }%
}
\makeatother
\title[Lipschitz Dynamical Systems]{A Lipschitz version of $\lambda$-Lemma and a characterization of homoclinic and heteroclinic orbits}

\numberwithin{equation}{section} 
\numberwithin{theorem}{section}

\author[G. La Guardia]{Giuliano G. La Guardia}
\address[Giuliano G. La Guardia]{Departamento  de Matem\'atica e Estatística,
Universidade Estadual de Ponta Grossa, 84030-900 Ponta Grossa PR, Brazil}
\email{gguardia@uepg.br}

\author[L. Pires]{Leonardo Pires}
\address[L. Pires]{Departamento de Matem\'atica e Estatística, Universidade Estadual de Ponta Grossa, 84030-900, Ponta Grossa -PR, Brasil.}
\email{lpires@uepg.br}
\begin{document}

\begin{abstract}
In this paper we consider finite dimensional dynamical systems generated by a Lipschitz function.  We prove a version of the Witney's Extension Theorem on compact manifold  to obtain a version of the Lambda Lemma for Lipschitz functions.  The notions of Lipschitz transversality and hyperbolicity are studied in the context of finite dimension with a  norm weaker than $C^1$-norm  and stronger than $C^0$-norm. 
\end{abstract}

\maketitle

\section{Introduction}
Theory of differentiable dynamical systems is widely studied in several works where the dynamics of a $C^1$ diffeomorphism on a compact manifold $M$ is  considered. A lot of results are obtained in  the space  $\tn{Diff}(M)$ of all diffeomorphism endowed with the $C^1$-norm as we can see in \cite{Palis1982,shub,Smale_diff,Wen}. Another important class is Hom$(M)$ of dynamical systems generated by a homeomorphism  with the $C^0$-norm \cite{pilyudin1994}. In those classes two dynamical systems are essentially the same if there is a conjugation between them, i.e., a homeomorphism that 
maps trajectories preserving all fixed and periodic points. The set of fixed points  and  points that have a special type of recurrence as periodic points  is called nonwandering set and it follows from Tubular Flow Theorem that it is sufficient to compare dynamics locally in nonwandering set and connection between its elements. Therefore the permanence of the nowandering points under $C^1$ small perturbations is central in several works, for example \cite{Peixoto,Palis1968,Palis1070}. The property of preserving the dynamics is called  structural stability and  it is related with the conditions  Axiom A and strong transversality see \cite{Peixoto,Palis1070}, but both concepts require differentiability, thus the dynamical systems generated by continuous functions that have no differentiability had little development from the point of view of structural stability due to the absence of geometrical concepts such as transversality and hyperbolicity.

A consequence of Mean Value Theorem is  that a differentiable function in the Euclidian spaces is locally Lipschitz. On the other hand  the Rademacher's Theorem  ensure that a Lipschitz function is differentiable almost everywhere with respect to Lebesgue measure \cite{Juha_2004}. Therefore the class of invertible Lipschitz functions with inverse  also Lipschitz seems to be closed to diffeomorphisms' class. In fact, as we will see, there is many interesting geometric property from dynamical  system generate by a Lipschitz function not differentiable on Riemannian compact smooth manifold. The space Lip$(M)$ of Lipschitz maps has a natural norm considered in several works in order to get approximation and extension results in nonlinear functional analysis \cite{GeraldBeer,Cobzas2019,Garrido_2013,Juha_2004}. In \cite{lip_banach} the authors introduced the notion of hyperbolicity and transversality for Lipschitz function in Banach spaces to deal with a nonlinear evolution equation whose solution defines an infinity dimensional semi dynamical system which is Lipschitz but never differentiable.

In this paper we  propose a setting of Lipschitz dynamical system in finite dimension. We will show that  some of the results which are valid to discrete standard smooth dynamical systems also hold when considering a class of Lipschitz function instead of class of differentiable function. Our results strongly depend on finite dimension and sometimes on  compactness of space, these assumptions are essentials to prove that a Lipschitz function on a connected compact manifold can be approximated by a smooth function in the norm of Lip$(M)$. With this theorem and with the notion of Lipschitz hyperbolicity and transversality we will obtain a version of the Lambda Lemma for Lipschitz function and then we will study the chaotic behavior in a neighborhood of a homoclinic tranverse point. We also deal with the transitivity of points heteroclinically related.

Before we begin our approach we need to ensure that our results are indeed relevant. For this we point some  facts and examples to support our statements. The Lipschitz version of the Tubular Flow Theorem was proved in \cite{Craig_2008} thus, we will focus our attention on nonwandering set. The existence of not smooth Lipschitz function is very known, for example the distance function,  but we will consider more interesting examples. In \cite{Giovanni_2010} the authors proved that for each Lebesgue measure null set in $\R^2$ there is a Lipschitz function not differentiable in this set, thus it is possible to obtain it for example in a periodic orbit or homoclinic point. In fact, take $x_0\in \R$ and open connected sets $G_1\supset G_2\supset ...\supset\{x_0\}$ sufficiently small for that $G_k$ has small measure in  $G_{k-1}$, define $f_k(x)=|(-\infty,x)\cap G_k|$, then $f_k^\prime(x)=1$ for all $x\in G_k$ but the slop is close to zero, thus $f(x)=\sum(-1)^kf_k(x)$ is not differentiable at $\cap G_k=\{x_0\}$ and so $g(x)=f(x_0-x)-f(x_0)$ at the fixed point $x_0$. It's interesting to know what kind of dynamics $g$ has near $x_0$. We will address this kind of question in Section \ref{One dimensional dynamics}, we will see there that in this case the fixed point $x_0$ have behavior as a sink. Another interesting example appears when we consider the equation $\ddot{x}-x+\varphi=0$, where $\varphi$ is Lipschitz not differentiable at $0=\varphi(0)$ and $1=\varphi(1)$ with $\int_0^1\varphi=1/2$.   The energy $E(x)=\dot{x}^2/2+x^2/2+F(x)$, where $F(x)=\int_0^x\varphi$ is constant on each solution and then the set $\{(x,y):x>0, y^2=x^2-2F(x),0<x<1\}$ should be an important one in the nonwandering set, for example if $(1,0)$ is a center  this set would be a homoclinic orbit, but how can we conclude this, once $\varphi$ is not differentiable at $x=1$? To deal with dimension greater than one we consider as in previously example functions of type $A+\varphi$ where $A$ is a linear operator and $\varphi$ is Lipschitz with a small constant but not necessarily Lip$(\varphi)$ goes to zero when we shrink the domain (this implies differentiability). The control in Lip$(\varphi)$ involves the norm and mininorm of the operator $A$ it is exactly what we need to obtain the conditions of the standard  stable  and unstable manifold theorem \cite[Chapter 2]{Wen}. Thus one Lipschitz map which is not differentiable should produce interesting dynamics even if we start at point of non differentiability or if a fixed point is one point for which the differentially fails. 

In order to carry out our statement, we divided this paper in the following way. In Section \ref{Lipschitz norm} we define the Lipschitz norm and the concept of transvesality,  also explaining how it is possible to work in a compact manifold,  we present one version of the Witney's Extension Theorem  which states that in a compact manifold every Lipschitz function has a extension whose Lipschitz constants can be made arbitrarily close to the Lipschitz constant of the original function.  In Section \ref{Lipschitz hyperbolic fixed points} we deal with Lipschitz perturbation of an isomorphism, that is, function of type $A+\varphi$ where $A$ is a linear  operator and $\varphi$ is Lipschitz with small constant, we state conditions to obtain the Hartman-Grobmann's Theorem and the stable manifold theorem. Moreover we study the definition of Lipschitz hyperbolicity presented in \cite{lip_banach}. In Section \ref{Homoclinic and heteroclinic orbits} we obtain the main result of this work, a Lipschitz version of the Lambda Lemma of \cite{Palis1982}. The Section \ref{One dimensional dynamics} is exclusively for one dimensional dynamics. We improve and extend the result in \cite{Giuliano} in order to see the main features of maps defined in $\R$. We finish our paper with the Section \ref{Lyapunov Exponent} where we propose a definition of Lyapunov exponent for locally Lipschitz maps, which extends, in a natural way, the well known concept of Lyapunov exponent for differentiable maps. 

\section{Lipschitz norm}\label{Lipschitz norm}
In this section we define the Lipschitz norm and the concept of transversality with no differentiability requirement. When we work on compact manifold each Lipschitz function has a $C^1$ extension sufficiently closed in the Lipschitz norm. Therefore in this section we fixed the notation and states important results that will be used in the next sections. 
\subsection{Lipschitz norm and Transversality}
Let $(E,|\cdot|)$ be a  finite-dimensional normed vector space. We define the {\it Lipschitz seminorm} $\tn{Lip}(f)$ of a Lipschitz function $f:E\to E$ as the least Lipschitz constant for $f$, i.e.,
\begin{equation}\label{lip_seminorm}
\tn{Lip}(f):= \inf\{L > 0: L \tn{ is a Lipschitz constant for }f\}.
\end{equation}
The set of all  Lipschitz function is a vector spaces with respect to the pointwise operations of addition and multiplication by scalars [Lip$(f+g)\leq \tn{Lip}(f)+\tn{Lip}(g)$ and Lip$(\alpha f)=|\alpha|\tn{Lip}(f)$]. 
But \eqref{lip_seminorm} not define a norm in this space, in fact Lip$(f)=0$ implies that $f$ is constant. Of course we can consider $f(0)+\tn{Lip}(f)$  that yields a norm but we won't do that. After  we will use a more interesting strategy to make Lip$(f)$ a norm. [Lip$(f)$ is a norm in the space of the Lipschitz functions that satisfy $f(0)=0$]. Some simple properties of the Lipschitz seminorm are:
\begin{itemize}
\item[(i)] $|f(x)-f(y)|\leq\tn{Lip}(f)|x-y|$, for all $ x,y\in  E$.
\item[(ii)] If $L$  is a Lipschitz constant for $f$, then $\tn{Lip}(f)\leq L$.
\item[(iii)] It is valid the following characterization
\begin{equation}\label{characterization_Lip_Norm}
\tn{Lip}(f)=\sup\Big\{\frac{|f(x)-f(y)|}{|x-y|}:x,y\in E,x\neq y \Big\}.
\end{equation}
\end{itemize}
Recall that  a function $f:E\to E$ is called {\it locally Lipschitz} if for every $x\in E$ there exists a neighborhood $U_x$, of $x$ such that $f|_{U_x}$ is Lipschitz. Since continuity is a local property, a  locally Lipschitz function is continuous. There exist locally Lipschitz functions which are not Lipschitz as well. But any locally Lipschitz function is Lipschitz on every compact subset of $E$. We denote the Lipschitz constant of $f$ in $U_x$ by
$
\tn{Lip}(f,U_x).
$

In case of differentiable functions we have the following simple characterizations of the Lipschitz property.

\begin{proposition}\label{MVT}
Let $U$ be an open subset of $E$ and let $f:U\to E$ be a continuously differentiable function. Then $f$ is locally  Lipschitz, more precisely  if $x_0\in U$ and $r>0$  are such that $B_r(x_0)\subset \overline{B_r(x_0)}\subset U$, then 
$$
\tn{Lip}(f,B_r(x_0))=\sup\{|f^\prime(z)|:z\in B_r(x_0)\}.
$$
\end{proposition}
\begin{proof}
Let $x_0\in U$ and $r>0$ such that $B_r(x_0)\subset \overline{B_r(x_0)}\subset U$. Then for all $x,y\in B_r(x_0)$, by Mean Value Theorem we have
$
|f(x)-f(y)|\leq \sup\{|f^\prime(z)|:z\in B_r(x_0)\} |x-y|,
$
hence Lip$(f,B_r(x_0))\leq\sup\{|f^\prime(z)|:z\in B_r(x_0)\}$ and $f$ is locally Lipschitz. Moreover, for $x,y\in B_r(x_0)$, $x\neq y$ we have 
$$
\frac{|f(x)-f(y)|}{|x-y|}\leq \tn{Lip}(f),
$$
taking $y\to x$ we obtain $|f^\prime(x)|\leq \tn{Lip}(f,B_r(x_0))$ for all $x\in B_r(x_0)$.
\end{proof}

Given a locally Lipschitz map $f:E\to E$ we define for $x_0\in E$ the {\it local Lipschitz norm} 
\begin{equation}\label{Lip_Norm}
\|f\|_{\tn{Lip},U}=\max \{ \|f\|_{C^0,U},\tn{Lip}(f,U)\},
\end{equation}
where $U$ is a bounded neighborhood of $x_0$ and $\|f\|_{C^0,U}=\sup\{|f(x)|:x\in \bar{U}\}$. Sometimes it is convenient to take $U=U_r=B_r(x_0)$, for some $r>0$. Of curse $\|f\|_{\tn{Lip},U}$ is not a norm in $E$ in the usual sense, since $f$ can be unbounded, in the next subsection we will replace $U$ by a compact manifold where in fact we will have a norm. The space of Lipschitz bound function with the norm \eqref{Lip_Norm} was considered in several works as \cite{GeraldBeer,Cobzas2019,Movahedi-Lankarani2003,D.Sherbert} each one with its specific goals. The local version \eqref{Lip_Norm} was presented in \cite{lip_banach}. If $f$ is a locally lipeomorphism (inverse is also locally Lipschitz) then we define
\begin{equation}
\|f\|_{\tn{Lip},U}=\max\{\|f\|_{C^0,U},\tn{Lip}(f,U),\|f^{-1}\|_{C^0,U},\tn{Lip}(f^{-1},U)\}.
\end{equation}
Notice that if two functions are close to one another in the Local Lipschitz norm then not only must the graphs be uniformly close but all inclinations of the secant lines must be close in the neighborhood considered. In fact the next lemma states that if eventually one of the functions is differentiable then we can change the Newton quotient by supremum of derivative in the neighborhood. Thus the local Lipschitz norm seems to be as close as possible to the $C^1$-norm if we do not require differentiability.

\begin{proposition}
Let $f$ be a Lipschitz map in a neighborhood $U$ of a point $p$. For any $\varepsilon>0$ there are $\delta,r>0$ such that if $g$ is  continuosly differentiable in $U$ such that $\|g-f\|_{\tn{Lip},U}\leq \delta$ then $|g'(x)-\tn{Lip}(f)|\leq \varepsilon$ on $B_r(p)\subset U$. 
\end{proposition}
\begin{proof}
The proof is immediate from Proposition \ref{MVT} since $g$ is differentiable then $g$ is locally Lipschitz with constant given by supremum of the derivative, hence the Lipschitz constants are sufficiently closed. 
\end{proof}

The notion of transversality play an important role in differentiable dynamical systems, it is an essential assumption to obtain structural stability (\cite{Palis1968,Palis1070,Peixoto}). Recall that two differentiable submanifolds $S_1,S_2\subset E$ are {\it transverse  at} $x_0\in S_1\cap S_2$ denoted by $S_1\tpitchfork_{x_0} S_2$ when the tangent spaces at $x_0$ generates all $E$, that is, $T_{x_0}S_1\oplus T_{x_0}S_2=E$. They are transverse if it are in all $x\in S_1\cap S_2$. Since the above definition evolves the tangent spaces and transversality can be characterized by  inclusion maps, the assumption's differentiability is central here. There is a notion of tangent space for for topological submanifolds but this definition may not be a vector space. On the other hand, if we think a submanifold as locally graph (in fact it is) it is intuitive that we can draw transverse graphs  even when they do not represent differentiable functions. In fact in \cite{lip_banach} the authors proposed the following definition of Lipschitz tranversality for graphs of Lipschitz functions.

We say that two no-empty sets $W_1,W_2\subset E$ are {\it $L-$transversal at $x\in W_1\cap W_2$} if there exist vector subspaces $E_1,E_2\subset E$, with $E=E_1\oplus E_2$, a real number $r>0$ and two Lipschitz  functions $\theta\colon B_r^{E_1}(0)\to E_2$ and $\sigma\colon B_r^{E_2}(0)\to E_1$, with $\theta(0)=\sigma(0)=0$, ${\rm Lip}(\theta)<1$, ${\rm Lip}(\sigma)<1$ and
$$
\{x+\xi+\theta(\xi)\colon \xi\in B_r^{E_1}(0)\}\subset W_1 \quad \hbox{ and }\quad \{x+\sigma(\eta)+\eta\colon \eta \in B_r^{E_2}(0)\}\subset W_2.
$$
We denote it by $W_1\tpitchfork_{L,x} W_2$. If $W^1$ and $W^2$ are $L$-transversal for every $x\in W^1\cap W^2$, we say that $W^1$ and $W^2$ are $L$-transversal and denote it by $W^1\tpitchfork_{L} W^2$. Notice that as in the definition of transversality  if $W^1\cap W^2=\emptyset$ then $W^1$ and $W^2$ are $L$-transversal, by vacuity.

We can also see that if $S_1\subset E$ is a differentiable submanifold then $S_1$ is locally a graph of a $C^1$ map $\theta :E_1\to E_2$ such that, if dim$(S_1)=n_1$ then dim$(E_1)=n_1$, dim$(E_2)=n_2$  where dim$(E)=n_1+n_2$,  $E=E_1\oplus E_2$. Thus if $S_2\subset E$ is a differentiable submanifold transverse to $S_1$ at a point $x$, then $S_2$ is also  locally a graph of a $C^1$ map $\sigma :\tilde{E}_2\to \tilde{E}_1$ such that dim$(S_2)=n_2$, dim$(\tilde{E}_2)=n_2$, dim$(\tilde{E}_1)=n_1$ and we can identify $E_1=\tilde{E}_1=T_xS_1$ and $E_2=\tilde{E}_2=T_xS_2$, i.e, we can take $S_1$ and $S_2$ as complementary  locally graphs under their respective tangent space at transversality point $x$.  Moreover, since $\theta$ and $\sigma$ are $C^1$ and the submanifolds $S_1,S_2$ are tangent to $E_1,E_2$ respectively at $x$ then we can reduce the the neighborhood of $x$ in order to take the Lipschitz constants of $\theta,\sigma$ smaller than one. Hence with these identification we have transversality implies $L$-transversality.

The most important property about transversality is the stability under small $C^1$ perturbation. The next theorem is a Lipschitz version of this fact. 

\begin{theorem}[Openness of transversality]\label{Openness}
Let $E$ be a finite dimensional normed space having a splitting $E=E_1\oplus E_2$. Assume that there exist $r>0$ and Lipschitz functions $\theta,\tilde{\theta}\colon B_r^{E_1}(0)\to E_2$, $\sigma,\tilde{\sigma}\colon B_r^{E_2}(0)\to E_1$ with $\theta(0)=\sigma(0)=0$, ${\rm Lip}(\theta)<1$ and ${\rm Lip}(\sigma)<1$. Define the sets:
\begin{align*}
&M=\{y+\theta(y)\colon y\in {B_r^{E_1}(0)}\},\quad \quad\quad\quad N=\{\sigma(x)+x\colon x\in {B_r^{E_2}(0)}\}, \\
&
\tilde{M}=\{z+y+\tilde{\theta}(y)\colon y\in {B_r^{E_1}(0)}\} \hbox{ and } \ \tilde{N}=\{z+\tilde{\sigma}(x)+x\colon x\in {B_r^{E_2}(0)}\},
\end{align*}
where $z\in E$, and suppose also that there exists $0<c<1$ such that ${\rm Lip}(\theta)\leq c$, ${\rm Lip}(\sigma) \leq c$ and
\begin{equation} \label{eq:LipTodas}
  \|\theta(y)-\tilde{\theta}(y)\| \leq (1-c)\tfrac r2 \hbox{ and } \|\sigma(x)-\tilde{\sigma}(x)\|\leq (1-c)\tfrac r2 \hbox{ for all } y\in {B_r^{E_1}(0)}, \ x\in {B_r^{E_2}(0)}.
\end{equation}
If ${\rm Lip}(\tilde{\theta})<1$ and ${\rm Lip}(\tilde{\sigma})<1$ then there exists a point $y_0$ such that $\tilde{M}\tpitchfork_{L,y_0} \tilde{N}$.
\end{theorem}
\begin{proof}
Let $K^1_r=\overline{B_{r/2}^{E_1}(0)}$ and $K^2_r=\overline{B_{r/2}^{E_2}(0)}$. Using \eqref{eq:LipTodas}, for $y\in K^1_r$ we have
$$
\|\tilde{\theta}(y)\|\leq \|\tilde{\theta}(y)-\theta(y)\|+\|\theta(y)\| \leq (1-c)\tfrac r2+c \tfrac r2  = \tfrac r2,
$$
and hence $\tilde{\theta}(K^1_r)\subset K^2_r$. Analogously, we obtain $\tilde{\sigma}(K^2_r)\subset K^1_r$.

We see that $\tilde{M}$ and $\tilde{N}$ have non-empty intersection if there exist $y\in K^1_r$ and $x\in K^2_r$ such that $y+\tilde{\theta}(y)=\tilde{\sigma}(x)+x$. The latter is true if there exists $y\in K^1_r$ such that $\tilde{\sigma}(\tilde{\theta}(y))=y$; that is, the map $g\colon K^1_r\to K^1_r$ given by $g(y)=\tilde{\sigma}(\tilde{\theta}(y))$ has a fixed point.

Clearly the map $g$ is well-defined, for if $y\in K^1_r$ then
$$
\|g(y)\| = \|\tilde{\sigma}(\tilde{\theta}(y)) - \sigma(\tilde{\theta}(y))\|+\|\sigma(\tilde{\theta}(y))-\sigma(0)\| \leqslant (1-c)\tfrac r2+c\tfrac r2 = \tfrac r2.
$$
By Brouwer's Fixed Point Theorem implies that $g$ has a fixed point in $K^1_r$ and hence $\tilde{M}\cap \tilde{N}\neq \emptyset$. Note that
$
\|g(y_1)-g(y_2)\| = \|\tilde{\sigma}(\tilde{\theta}(y_1))-\tilde{\sigma}(\tilde{\theta}(y_1))\| \leqslant {\rm Lip}(\tilde{\sigma}) {\rm Lip}(\tilde{\theta})\|y_1-y_2\|,
$
for all $y_1,y_2\in K^1_r$. Therefore $g$ is a contraction and has a unique fixed point $y_1\in K^1_r$ and defining $y_2=\tilde{\theta}(y_1)$ we have $y_0=y_1+y_2\in \tilde{M}\cap \tilde{N}$. Note that, by construction, we also have $y_1=\tilde{\sigma}(y_2)$. It remains to prove the $L$-transversality at $y_0$. To this end, firstly we choose $r_0>0$ such that $\overline{B_{r_0}^{X_1}(y_1)}\subset B_r^{X_1}(0)$ and $\overline{B_{r_0}^{X_2}(y_2)}\subset B_r^{X_2}(0)$. Now we define functions $\theta_\ast \colon B_{r_0}^{X_1}(0)\to X_2$ and $\sigma_\ast \colon B_{r_0}^{X_2}(0)\to X_1$ by $
\theta_\ast(y) = \tilde{\theta}(y+y_1) - y_2 = \tilde{\theta}(y+y_1) - \tilde{\theta}(y_1)$ and $\sigma_\ast(x) = \tilde{\sigma}(x+y_2)-y_1 = \tilde{\sigma}(x+y_2)-\tilde{\sigma}(y_2)$ for all $y\in B_{r_0}^{X_1}(0)$, $x\in  B_{r_0}^{X_2}(0)$. Therefore, the functions $\theta_\ast$ and $\sigma_\ast$ satisfy the $L$-transversal conditions, and hence $\tilde{M}\tpitchfork_{L,y_0} \tilde{N}$.
\end{proof}

\subsection{Lipschitz norm on compact manifold}\label{Lipschitz norm on compact manifold}
let $M$ be a compact and connected  $C^\infty$ Riemann manifold without boundary with a Riemannian metric $d$ embedded in a finite dimensional euclidian space $(E,|\cdot |)$. Since $M$ is compact every locally Lipschitz function admits a Lipschitz extension to $M$ hence we just consider globally Lipschitz functions on $M$. In addition we consider the Lebesgue measure acting in $M$, i.e., a null set in $M$ are those sets that are null set in $E$ with to respect to Lebesgue measure. As in \cite{Kentaro_1974} we assume  $\sup\{d(f(x),x):x\in M\}$  uniformly bounded for all Lipschitz functions $f$ on $M$ then a function that is Lipschitz with respect to $d$ is also Lipschitz  with respect to the Euclidean metric. Therefore  the Rademacher's Theorem applies and states that a Lipschitz function $f$ is differentiable a.e. on $\R^n$ hence differentiable a.e. on $M$.
In what follows we denote $\tn{Lip}(M)$ the set of all lipeomorphism of $M$ into itself with the norm 
\begin{equation}
\|f\|_{\tn{Lip}}=\max\{\|f\|_{C^0},\tn{Lip}(f),\|f^{-1}\|_{C^0},\tn{Lip}(f^{-1})\}.
\end{equation}
We denote $C(M)$ the set continuous function of $M$ into itself with $C^0$ norm. It is known that  the space $C(M)$ is separable if $M$ is a compact.  
Moreover the separability can be done with Lipschitz functions (see \cite{Cobzas2019}), i.e.,  the space $C(M)$ contains a countable dense subset formed of Lipschitz functions. We also consider $C^1(M)$ the set of all continuosly differentiable functions of $M$ into itself with $C^1$-norm  $\|f\|_{C^1}=\|f\|_{C^0}+\|f^\prime\|_{C^0}$.  It is clear that $C^1(M)\subset \tn{Lip}(M)$, more generally it follows from definition \eqref{lip_seminorm} and characterization \eqref{characterization_Lip_Norm} the following result.
\begin{proposition}Let $M$ be a compact $C^\infty$ Riemann manifold without boundary. Denote $\tn{Hom}(M)$ and $\tn{Diff}^1(M)$ the sets of all homeomorphisms and diffeomorphisms of $M$ into itself respectively. The  following inclusions are true
$\tn{Diff}^1(M)\hookrightarrow \tn{Lip}(M)\hookrightarrow \tn{Hom}(M)$ that is,
\begin{itemize}
\item[(i)] $\|f\|_{\tn{Lip}}\leq \|f\|_{C^1}$, for all $f\in\tn{Diff}^1(M)$.
\item[(ii)] $\|f\|_{C^0}\leq \|f\|_{\tn{Lip}}$, for all $f\in\tn{Lip}(M)$.
\end{itemize}
\end{proposition}
An interesting question is whether $C^1(M)$ is closed in  $\tn{Lip}(M)$, when $M$ is compact. For infinite dimension the negative answer is given in \cite{Movahedi-Lankarani2003} where the authors exhibit an intermediate ring of polynomial functions between $C^1(M)$ and Lip$(M)$. But  if the dimension is  finite  the answer is (almost) affirmative, i.e., except for a set with small measure. In fact by $C^1$-Witney's Extension Theorem see \cite[Theorem 1 section 6.5]{Evansc} we have that if $f$ is a Lipschitz function, then $f$ admits a $C^1$ extension  $g$ such that $f=g$ and $Df=Dg$ except in a  set with small measure \cite[Theorem 1 subsection 6.1.1]{Evansc}. We will adapt the Witney's Theorem to ensure that in a compact manifold the extension $g$ is close to $f$ in Lipschitz norm. A more general version of this result for second countable Finsler manifold can be found in \cite{Garrido_2013}.

\begin{theorem}\label{Witney_extension}
Let $f: M\to M$ be a Lipschitz function,  then for $\varepsilon>0$ there is a $C^1$ function $g: M\to M$ such that $\|f-g\|_{\rm Lip}\leq \varepsilon$. 
\end{theorem}
\begin{proof}
First consider $f:M\to \mathbb{R}$. By  McShane-Withey's extension Theorem \cite[Theorem 2.3]{Juha_2004} we extend $f$ to $\mathbb{R}^n$ and by  Rademacher's Theorem \cite{Evansc,Juha_2004}, $f$ is differentiable on a set $A\subset \R^n$ such that $|\R^n\bac A|=0$ and then $|(\R^n\bac A)\cap M|=0$. Now by Lusin's Theorem \cite{Evansc,Cobzas2019}, there is a set $B\subset A\cap M$ such that $Df|_B$ is continuous and $|\R^n\bac B|< \varepsilon/2$. Set 
$$
R(x,y)=\frac{f(y)-f(x)-Df(x)(y-x)}{|x-y|},\quad x\neq y,\,\,x,y\in B, 
$$
$$
\eta_k(x)=\{R(x,y):y\in B,\,\, 0<|x-y|<1/k\}.
$$
Then by Ergoroff's Theorem there is a set $C\subset B$ such that $\eta_k\to 0$ uniformly on compact subsets of $C$ and $|B\bac C|\leq \varepsilon /2$. It follows from Witney's Extension  Theorem the existence of a $C^1$ extension $g$ of $f$ in $\R^n$ and then in $M$. In fact $g$ is given by
$$
g(x)=\begin{cases} f(x)& x\in C\\ \sum_{j=1}^m v_j(x)[f(s_j)+Df(s_j)(x-s_j)] & x\in M\bac C, 
\end{cases} 
$$
where $s_j \in C$ for $j=1,...,m$, $|x_j-s_j|=\tn{dist}(x_j,C)$ and  $\{x_j\}_j$ is given by Vitali's Covering Theorem is such that $M\bac C=\cup_j B_{r(x_j)}(x_j)$ for appropriate $r(x_j)>0$ moreover $\{v_j\}_j$ are smooth partition of unit in $M\bac C$. Thus if $x\in M\bac C$ then $x\in B_{r(x_i)}(x_j)$ for some $j$. Since $|M\bac C|\leq \varepsilon$ we can take $|x_j-s_j|\leq \varepsilon$ and for simplicity we assume $v_j=1$, thus 
\begin{align*}
|f(x)-g(x)|\leq |f(x)-f(s_j)|+|Df(x)||x-s_j|\leq \tn{Lip}(f)|x-s_j|+|Df(x)||x-s_j|\leq C\varepsilon.
\end{align*}
Moreover $g$ is continuosly differentiable and its derivative is given by
$$
Dg(x)=\sum_{x_j\in S_x}\{[f(s_j)+Df(s_j)(x-s_j)]Dv_j(x)+V_j(x)Df(s_j) \},
$$
where $S_x=\{x_j: B_{10r(x)}(x)\cap B_{10r(x_j)}(x_j)\neq\emptyset\}$. If $x\in M\bac C$ then $x\in B_{r(x_i)}(x_j)$ for some $j$ and for simplicity we assume $v_j=1$ and $Dv_j=0$, thus we have
\begin{align*}
|f(x)-f(x_j)-Dg(x)| &\leq |f(x)-f(s_j)+f(s_j)-f(x_j)-f(s_j)Df(s_j)(x-x_j)|\\
&+ |Df(s_j)(x-s_j)(x-x_j)|\\
&\leq \tn{Lip}(f)|x-x_j|+\tn{Lip}(f)|s_j-x_j|+C|x-s_j|\leq C\varepsilon.
\end{align*}
Rescaling $\varepsilon$ we obtain $\|f-g\|_{\rm Lip}\leq\varepsilon$. To finish the theorem notice that a function $f:M\to \R^n$ is Lipschitz if only if its coordinates function are Lipschitz so the theorem is true for  $f: M\to\R^n$. Finally if $f: M\to M$ is Lipschitz then for each $x\in M$ take a chart $u$ and  then apply the theorem for $u\circ f: M\to \R^n$ in order to obtain a extension $g: M\to \R^n$, thus $u^{-1}\circ g$ is a extension of $f$ in chart that can be extended to $M$ by Gluing Lemma for smooth function.  
\end{proof}

The Theorem \ref{Witney_extension}  will be the main tool in Section \ref{Homoclinic and heteroclinic orbits}  to transfer important results about differentiable dynamical systems  to Lipschitz dynamical systems.

\section{Lipschitz hyperbolic fixed points}\label{Lipschitz hyperbolic fixed points}
In this section we define the concept of Lipschitz hyperbolicity. We will  consider Lipschitz perturbation of a hyperbolic isomorphism, that is, function of type $A+\varphi$, where $A$ is an isomorphism and $\varphi$ is Lipschitz with small constant controlled by norm of $A$. In fact we present conditions to obtain the standards theorems of smooth dynamical system to our context. Here we consider that the dynamics occurs in a phase space of the dimension greater than one. The one dimensional case will be presented separately in Section \ref{One dimensional dynamics}.       

Let $(E,|\cdot|)$ be a  finite-dimensional normed vector space with dim$(E)>1$ and denote $\mathcal{L}(E)$ the set of all linear operator of $E$. For $A\in\mathcal{L}(E)$ we define the operator norm and mininorm of $A$ respectively by 
$
\|A\|_{\mathcal{L}}=\sup\{|Av|:|v|=1 \}\mbox{ and }  m(A)=\inf\{|Av|:|v|=1\}. 
$
Recall that a linear isomorphism $A\in\mathcal{L}(E)$ is called {\it hyperbolic} if there is a splitting of $E=E^s\oplus E^u$ such that 
$A(E^s)=E^s$ and $A(E^u)=E^u$ and there are constants $c>0$ and $\lambda\in (0,1)$ such that the following uniform exponential estimates hold
$$
|A^n u|\leq C\lambda^n |u|,\,\forall \,u\in E^s,\,n\geq 0,
$$
$$
|A^{-n} u|\leq C\lambda^n |u|,\,\forall \,u\in E^u,\,n\geq 0.
$$
We denote the {\it skewness} of $A$ by $\tau(A)=\max\{ \|A|_{E^s}\|_{\mathcal{L}},\|A^{-1}|_{E^u}\|_{\mathcal{L}}\}<1$.
\begin{definition}
Let $f:E\to E$ be a map. We say that a fixed point $p$ of $f$ is {\it Lipschitz hyperbolic} (L-hyperbolic) if there is a neighborhood $U$ of $p$ such that $f$ decompose $f=A+\varphi$ in $U$, where $A\in \mathcal{L}(E)$ is a hyperbolic linear isomorphism with splitting $E=E^s\oplus E^u$ and $\varphi$ is Lipschitz continuous such that
\begin{equation}\label{Lip_estimates}
\textnormal{Lip}(\varphi)<\min\Big\{\textstyle\dfrac{1-\tau(A)}{2},m(A)\Big\}.
\end{equation}
\end{definition}
In this case, for $r>0$ we define the {\it local and global stable manifold} respectively by
$$
W^s_r(p,f)=\{x\in U: |f^n(x)-p|\leq r,\textnormal{ and } \lim_{n\to\infty} f^n(x)=p \},
$$
$$
W^s(p,f)=\{x\in U:  \lim_{n\to\infty} f^n(x)=p \},
$$
if $f$ is invertible we define {\it local and global  unstable manifold} respectively by
$$
W^u_r(p,f)=\{x\in U: |f^{-n}(x)-p|\leq r,\textnormal{ and } \lim_{n\to\infty} f^{-n}(x)=p \}.
$$
$$
W^u_r(p,f)=\{x\in U:  \lim_{n\to\infty} f^{-n}(x)=p \}.
$$
We define the {\it  dimension} of stable and unstable manifolds as the dimension of $E^s$ and $E^u$ respectively. The proof of the next theorem is well known and can be found in \cite{hpg,katok,Robinson_chaos,shub,Wen}.

\begin{theorem}\label{stable_manifold_theorem}
If $p$ is a fixed L-hyperbolic fixed point of $f$ then the following statement are true:
\begin{itemize}
\item[(i)]  (Uniqueness of L-hyperbolic fixed point)  There is a neighborhood $U_r$ of $p$ such that if $|\varphi(0)|\leq(1-\tau(A)-{\rm Lip}(\varphi))r $ then $p$ is the unique fixed point of $f$ in $U_r$.
\item[(ii)]$A+\varphi$ is a Lipschitz homeomorphism with Lipschtz inverse satisfying \tn{Lip}$((A+\varphi)^{-1})\leq 1/[m(A)-\tn{Lip}(\varphi)]$.
\item[(iii)] (L-Hartman-Grobman theorem) There is a homeomorphism $h:U\to U$ that conjugates $A$ and $f$, that is  $h\circ f=A\circ h$. Therefore $W^s_r(p,f)$ and $W^u_r(p,f)$ are topological submanifold of $U$.
\item[(iv)] We have the following characterizations:
\begin{align*}
W^s_r(p,f)&=\{x\in U_r: |f^n(x)-p|\leq r, \forall n\geq 0\} \\
&=\{x\in U_r: |f^n(x)-p|\leq r\textnormal{ and }|f^n(x)-p|\leq\lambda^n|x-p|,  \forall n\geq 0  \},
\end{align*}
\begin{align*}
W^u_r(p,f)&=\{x\in U_r: |f^{-n}(x)-p|\leq r, \forall n\geq 0\} \\
&=\{x\in U_r: |f^{-n}(x)-p|\leq r\textnormal{ and }|f^{-n}(x)-p|\leq\lambda^n|x-p|,  \forall n\geq 0  \}.
\end{align*}
\item[(v)] (Isolation of L-Hyperbolic fixed point) If for $q\in U$ we have $f^n(q)-p\in U$  for all $n\in\mathbb{Z}$ then $q=p$.
\item[(vi)] (Global stable and unstable manifold theorem) There is a Lipschitz maps $\sigma:E^s+p\to E^u+p$ and $\theta:E^u+p\to E^s+p$  such that $\theta(p)=p=\sigma(p)$, \textnormal{Lip}$(\theta)<1$,\textnormal{Lip}$(\sigma)\leq 1$ and $W^s(p,f)\subset\{p+\sigma(\eta)+\eta:\eta\in E_s\}$,   $W^u(p,f)\subset\{p+\xi+\theta(\xi):\xi\in E_u \}$.  
\item[(vii)] (Local stable and unstable manifold theorem) The local manifolds  can also be characterized as in the previous item \tn{(vi)} as a graph of lipschtz maps in $U_r$.
\end{itemize} 
\end{theorem}
\begin{proof}
In the following we will recall some proof details in order to fix  notations that we will use later. It is important to observe that we will use only the estimates \eqref{Lip_estimates} and at no time we will require differentiability in $\varphi$. We also observe that we can consider the statements (iii) and (vii) in a global version, that is, $\varphi$ defined in all $E$ in fact if $\varphi$ is defined only in a neighborhood of $p$ we can take Lipschitz cutoff in order to extend $\varphi$ in all $E$ preserving the Lipschitz property.  Since $A$ is hyperbolic we fix the splitting $E=E^s\oplus E^u$, without loss of generality we assume $C=1$ (adapted norm)  and the  maximum norm in $E$. Let $r>0$ and denote $U=U_r=B_r(0)$.

(i) For $v\in E$ we have $(A+\varphi)v=v$ if only if $v_s=A_sv_s+\varphi_s(v)$ and $v_u=A_u^{-1}v_u-A_u^{-1}\varphi_u(v)$. This define a contraction $T$ with Lip$(T)\leq \tau(A)-$Lip$(\varphi)$ that maps $U_r$ into itself. Note that \eqref{Lip_estimates} ensure that  Lip$(T)<1$.

(ii) To see that $(A+\varphi)$ is injective take $v,z\in E$, then $(A+\varphi)v=z$  if only if $x=A^{-1}z+A^{-1}\varphi(x)$. This define a contraction $T$ with Lip$(T)\leq m(A)^{-1}$Lip$(\varphi)$. Moreover $|(A+\varphi)u-(A+\varphi)v|\geq [m(A)-\tn{Lip}(\varphi)]|u-v|$, for $u,v\in E$.

(iii) We have to find $h$ such that $h(A+\varphi)=Ah$. If we impose $h=I+\psi$, where $I$ is the identity in $E$, then $(I+\psi)(A+\varphi)=A(I+\psi)$ if only if $\varphi+\psi(A+\varphi)=A\psi$ if only if $\psi_s=(A_s\psi_s-\varphi_s)(A+\varphi)^{-1}$ and $\psi_u=A_u^{-1}(\varphi_u+\psi_u(A+\varphi))$. This define a contraction $\mathcal{T}$ in the space of bounded continuous map  with sup norm.  

(iv) We define the  $1$-cone about $E^s$ by
$
C_1(E^s)=\{v=(v_s,v_u) \in E : |v_u|\leq |v_s|\}. 
$
Without loss of generality we assume $p=0$, then  the characterization follows showing that 
\begin{align*}
W^s_r(0,f)&=\{x\in U_r: |f^n(x)|\leq r, \forall n\geq 0\} \\
&= \{x\in U_r: |f^n(x)|\leq r\tn{ and } f^n(x)\in C_1(E^s), \forall n\geq 0\}\\ 
&=\{x\in U_r: |f^n(x)|\leq r\tn{ and }|f^n(x)|\leq(\tau(A)+\tn{Lip}(\varphi))^n|x|,  \forall n\geq 0\}.
\end{align*}
In the same define the  $1$-cone about $E^u$ by
$
C_1(E^u)=\{v=(v_s,v_u) \in E : |v_s|\leq |v_u|\}. 
$
We have that
\begin{align*}
W^u_r(0,f)&=\{x\in U_r: |f^{-n}(x)|\leq r, \forall n\geq 0\} \\
&= \{x\in U_r: |f^{-n}(x)|\leq r\tn{ and } f^{-n}(x)\in C_1(E^u), \forall n\geq 0\}\\ 
&=\{x\in U_r: |f^{-n}(x)|\leq r\tn{ and }|f^{-n}(x)|\leq(\tau(A)+\tn{Lip}(\varphi))^n|x|,  \forall n\geq 0\}.
\end{align*}

(v) If $q\in U_r$ satisfies $|f^n(q)-p|\leq r$ for all $n\in\mathbb{Z}$ then $|q-p|\leq \lambda^{2n}|q-p|$.

(vi) Denote $\mathcal{F}=\{\sigma:E^u\to E^s:\sigma(0)=0,\,\,\tn{Lip}(\sigma)\leq 1 \}$ with the Lip norm Lip$(\sigma)$ since $\sigma(0)=0$ it define a norm that makes $\mathcal{F}$ a complete metric space. Define the graph transform $T:\mathcal{F}\to\mathcal{F}$ by $T\sigma=(A_s\sigma+\varphi_s(I_s+\sigma))(A_u+\varphi_u(I_u+\sigma))^{-1}$. With aid of \eqref{Lip_estimates} we can prove that $T$ is a well defined contraction in $\mathcal{F}$ .
\end{proof}

Now we present the main result of this section. Since the stable and unstable manifold are given by graphs, we denote $W^s_r(q_\eta)\to W^s_r(p)$ when the respective graphs are closed in the Lipschitz norm.

\begin{theorem}\label{convergence_of_manifolds}
Let $p$ be a fixed L-hyperbolic point of $f$ such that $f=A+\varphi$ in $U_r$ neighborhood of $p$. If $\varepsilon>0$ and $g$ is continuously differentiable in $U_r$  such that $\|g-f\|\leq \varepsilon$, then for $\varepsilon$ sufficiently small $g$ has a unique hyperbolic fixed point $q$ in $U_r$. Moreover if $\{g_\eta\}_\eta$ is a family of continuously differentiable maps  in $U_r$ such that $\|g_\eta-f\|_{\rm Lip}\to 0$ as $\eta\to 0$, then  $|q_\eta-p|\to 0$ and $W^s_r(q_\eta)\to W^s_r(p)$, $W^u_r(q_\eta)\to W^u_r(p)$ in the Lipschitz norm as $\eta\to 0$.
\end{theorem}
\begin{proof}
We assume $p=0$. Since $g$ continuosly differentiable  and we write $g=B+\psi$, where $B=g^\prime(0)$ and $\varphi$ is such that Lip$(\varphi)$ is sufficiently small to satisfy \eqref{Lip_estimates}, thus by Theorem \ref{stable_manifold_theorem} item (i) $g$ ha a unique fixed point $q$ in a neighbohood of zero. Moreover since $\|A-B\|_{\mathcal{L}}\leq \varepsilon$ it is well known that the hyperbolicity of $A$ transfer to $B$, hence $q$ is a hyperbolic point for $g$. Now denote $g_\eta=A_\eta+\varphi_\eta$ if $q_\eta$ is the hyperbolic fixed point of $g_\eta$  then by Theorem \ref{stable_manifold_theorem} item (i) $q_\eta$ is the fixed point of operator $v_{\eta,s}=A_{\eta,s}v_{\eta,s}+\varphi_{\eta,s}(v)$ and $v_{\eta,u}=A_{\eta,u}^{-1}v_{\eta,u}-A_{\eta,u}^{-1}\varphi_{\eta,u}(v)$. The convergence $|q_\eta-p|\to 0$ follows noticing that $\|g_\eta-f\|_{\rm Lip}\to 0$ as $\eta\to 0$ implies  $A_\eta\to A$ in operator norm and $\varphi_\eta\to \varphi$ in Lipschitz seminorm. For the last part we will present the proof for the unstable manifold. Again we can consider $p=0=q_\eta$ and the global manifolds. Besides that making a translation by a lipeomorphism we can consider $W^u(0)=E^u$ and to avoid excessive notation we consider $A_\eta=A$. The perturbed unstable manifolds are given by a family of maps $\sigma_\eta: E^u\to E^s$ with ${\rm Lip}(\sigma_\eta)<1$ such that $W^{\rm u,\eta}(0)=\{\xi+\sigma_\eta(\xi): \xi \in E^u\}$, for sufficiently small $\eta$, with $\sigma_0=\sigma=0$, and
\begin{equation*}
\left\{
\begin{aligned}
&h_{\eta}(\xi):=A_u\xi+\varphi_{\eta,u}(\xi+\sigma_\eta(\xi))\\
&\sigma_{\eta}(h_{\eta}(\xi))=A_s\sigma(\xi)+\varphi_{\eta,s}(\xi+\sigma_\eta(\xi))
\end{aligned}\right.\quad \hbox{ for }\xi \in E^u.
\end{equation*}
We have to proof that the norm $\|\sigma_\eta\|_{\rm Lip}\to 0$ as $\eta\to 0^+$. We know that $h_\eta$ is invertible and if we denote $\tau(A)=\tau$ and Lip$(\varphi)=\gamma$ we have
$$
|h_\eta(\xi_1)-h_\eta(\xi_2)|\geq (\tau-2\gamma)|\xi_1-\xi_2|, \hbox{ for all }\xi_1,\xi_2\in E^u \hbox{ and small }\eta.
$$
Since $\varphi_{0,s}(\xi)=0$ for each $\xi\in E^u$ we have
\begin{align*}
|\sigma_\eta(h_\eta(\xi))|&\leq \tau|\sigma_\eta(\xi)|+|\varphi_{\eta,s}(\xi+\sigma_\eta(\xi))-\varphi_{0,s}(\xi+\sigma_\eta(\xi))|\\
& \qquad +|\varphi_{0,s}(\xi+\sigma_\eta(\xi))-\varphi_{0,s}(\xi)|\\
&\leq (\tau+\gamma) \|\sigma_\eta\|_{C^0}+\|\varphi_{\eta}-\varphi_{0}\|_{C^0},
\end{align*}
thus we have $\|\sigma_\eta\|_{C^0} \leq \frac{1}{1-\tau-\gamma}\|\varphi_{\eta}-\varphi_{0}\|_{C^0}$ and therefore $\|\sigma_\eta\|_{C^0} \to 0$ as $\eta\to 0^+$.

Now for $r_\eta=\|\sigma_\eta\|_{C^0}$ and $K_\eta=\|\varphi_{0,s}\|_{C^0}+\tn{Lip}(\varphi_{\eta,s}-\varphi_{0,s})$ we have
\begin{align*}
|\sigma_\eta(h_\eta(\xi_1))-\sigma_\eta(h_\eta(\xi_2))| & \leq
|A_s(\sigma_\eta(\xi_1)-\sigma_\eta(\xi_2))+\varphi_{\eta,s}(\xi+\sigma_\eta(\xi))-\varphi_{\eta,s}(\xi_2+\sigma_\eta(\xi_2))|\\
& \leq \tau\;\tn{Lip}(\sigma_\eta)|\xi_1-\xi_2|+K_\eta|(1+\tn{Lip}(\sigma_\eta))|\xi_1-\xi_2|\\
& \leq [(\tau+K_\eta)\tn{Lip}(\sigma_\eta)+K_\eta]|\xi_1-\xi_2|\\
& \leq \dfrac{(\tau+K_\eta)\tn{Lip}(\sigma_\eta)+K_\eta}{\tau-2\gamma}|h_\eta(\xi_1)-h_\eta(\xi_2)|,
\end{align*}
and hence
$$
{\rm Lip}(\sigma_\eta) \leq \frac{K_\eta}{\tau^{-1}-\tau-2\gamma-K_\eta} \to 0 \hbox{ as } \eta\to 0^+.
$$
which concludes the result, since $K_\eta \to 0$ as $\eta\to 0^+$.
\end{proof}

\section{Homoclinic and heteroclinic orbits}\label{Homoclinic and heteroclinic orbits}

Throughout this section we will consider $M$ a Riemann manifold and $f$ a Lipschitz function as in Subsection \ref{Lipschitz norm on compact manifold}. For  a fixed point $p\in M$ of $f$, we denote $0_p$ the corresponding null vector in tangent space $T_pM$ of $M$ at $p$ and we define the {\it self-lifting} of $f$ by the function $F_f: T_pM\to T_pM$ where $F_f(v)=\exp_p^{-1}   f \exp_p(v)$, for all $v\in T_pM$.

\begin{definition}
Let $f:M\to M$ be a function. We say that a fixed point $p$ of $f$ is {\it Lipschitz hyperbolic} (L-hyperbolic) if there is a neighborhood $U$ of $0_p$ in $T_pM$ such that $F_f$ decompose as $F_f=A+\varphi$ in $U$, where $A\in \mathcal{L}(E)$ is a hyperbolic linear isomorphism with splitting $E=E^s\oplus E^u$ and $\varphi$ is Lipschitz continuous satisfying \eqref{Lip_estimates}.
\end{definition}

If $p$  is a $L$-hyperbolic fixed point for $f$ then we can apply the Theorem \ref{stable_manifold_theorem} for $F_f$ and then obtain the same properties for $f$ once we have the stable and unstable manifold as projection of the exponential map, that is $W_r^s(p,f)=\exp_p(W_r^s(0_p,F_f))$ and $W_r^u(p,f)=\exp_p(W_r^u(0_p,F_f))$.

The next theorem is a Lipschitz  version of $\lambda$-Lemma \cite[Lemma 7.1, Chapter 2]{Palis1982}.

\begin{theorem}\label{lambda_lemma} Let $p$ be a $L$-hyperboic fixed point of  $f$. If $D$ is a submanifold such that $D\tpitchfork_{L,x} W^s(p,f)$ for some point $x\in D\cap W^s(p,f)$ and $\tn{dim}(D)=\tn{dim}(W^u(p,f))$, then  for $\varepsilon>0$ there is  $n_0\in \N$ such that $f^n(D)$ is $\varepsilon$-close to $W^u(p,f)$ in the \tn{Lip} norm for $n\geq n_0$.
\end{theorem}
\begin{proof}
Since $p$ is $L$-hyperbolic then $f$ is Lipschitz in $M$. Given $\varepsilon>0$ by Theorem \ref{Witney_extension} there is $g\in C^1(M)$ such that $\|f-g\|_{\rm Lip}\leq \varepsilon$ and it follows from Theorem \ref{stable_manifold_theorem} that $g$ has a unique hyperbolic fixed point $q$ in a neighborhood of $p$.  If  $D$ is a submanifold such that $D\tpitchfork_{L,x} W^s(p,f)$ for some point $x\in D\cap W^s(p,f)$ we claim that $D\tpitchfork_{y_0} W^s(q,g)$ for some point $y_0\in D\cap W^s(q,g)$. In fact, we take the splitting $E=E^u\oplus E^s$ given by hyperbolicity of $p$ and make a translation of the splitting of $q$ so that we can write
$$
y+D=\{y+\sigma(\eta)+\eta: \eta \in E^u_r\},\quad W^s(p)=\{x +\xi+\theta(\xi): \xi \in E^s_r \}, 
$$
$$
x+D=\{x+\sigma(\eta)+\eta: \eta \in E^u_r\},\quad W^s(q)=\{y +\xi+\tilde{\theta}(\xi): \xi \in E^s_r \}, 
$$
where  $E^s_r, E^s_r$ denotes a neighborhood of the origin in $E^s$ and $E^u$ respectively, $y\in M$ and $\sigma:E^u_r\to E^s$, $\theta : E^s_r\to E^u$ and $\tilde{\theta} : E^s_r\to E^u$ are Lipschitz with constant smaller than $1$. It follows from Theorem \ref{Openness} that there is $y_0\in D\cap W^s(q,g)$ such that $D\tpitchfork_{y_0} W^s(q,g)$. 

Now we  apply the $\lambda$-Lemma in $g$ in order to conclude the existence of a submanifold $\tilde{D}$ and $n_0\in \N$ such that $g^n(D)$ contains $\tilde{D}$ that  is $C^1$ $\varepsilon$-closed to $W^u(q,g)$ for all $n>n_0$. Since $W^s(q,g) \tpitchfork_{q} W^u(q,g)$ and transversality is stable under $C^1$ small perturbation we have $\tilde{D}\tpitchfork_{z_0} W^s(q,g)$ for some $z_0\in g^n(\tilde{D})\cap W^s(q,g)$. But transversality implies $L$-transversality thus with the same argument above we can use again the Theorem \ref{Openness} to ensure the existence of a point $w_0\in g^n(\tilde{D})\cap W^s(p,f)$ such that $\tilde{D}\tpitchfork_{L,w_0} W^s(p,f)$. But $f$ is Lipschitz $\varepsilon$-closed to $g$ then $f^n(D)$  is Lipschitz $\varepsilon$-closed to $g^n(D)$ for $n\geq n_0$ and the Theorem \ref{convergence_of_manifolds} ensures that the unstable and stable manifolds are Lipschitz closed, hence $f^n(D)$  is Lipschitz $\varepsilon$-closed to $W^u(p,f)$.
\end{proof}
\begin{remark}
If $f$ is a lipeomorphism then the same statement of Theorem \ref{lambda_lemma} is valid for unstable manifold, that is, the backward evolution of a submanifold $L$-transverse to the unstable manifold will be sufficiently closed in Lipschitz norm to stable manifold.   
\end{remark}
Let $p$ be a $L$-hyperbolic point of $f$. We say $x\in E$ is a {\it L-homoclinic poin}t of $p$ if $x\neq p$ and  $x\in W^s(p,f)\cap W^u(p,f)$. A $L$-homoclinic point $x$ of $p$ is called $L$-{\it transverse} if $W^s(p,f)\tpitchfork_{L,x} W^u(p,f)$. We say that two $L$-hyperbolic fixed points are {\it $L$-homoclinically related} if there is $x,y$, $x\neq y$ such that $W^s(p,f)\tpitchfork_{L,x} W^u(q,f)$ and $W^s(q,f)\tpitchfork_{L,y} W^u(p,f)$. Recall that $f$ has chaotic behavior in a set $V$ if there is a compact $\Lambda\subset V$ and $m>0$ such that $f^m:\Lambda\to \Lambda$ is topologically conjugated to shift $\sigma_2:\Sigma_2\to \Sigma_2$ where $\Sigma_2=\prod_{-\infty}^\infty A_n $, $A_n=\{0,1\}$ for $n\in \mathbb{Z}$.

The next theorem is a Lipschitz version of the Birkhoff-Smale Theorem \cite{SMALEg}.

\begin{theorem} Let $p,q$ be  $L$-hyperbolic points of  $f$.
\begin{itemize}
\item[(i)] if there is a $L$-homoclinic transverse point $x$ of $p$ then $f$ has chaotic behavior in all neighborhood of $\{p,x\}$. 
\item[(ii)] if $p$ and $q$ are are $L$-homoclinically related then  $f$ has chaotic behavior in a neighborhood of $\{p,q\}$.
\end{itemize}
\end{theorem}
\begin{proof}
Since we have a Lipschitz version of the $\lambda$-Lemma the proof is the same as doing in \cite{SMALEg,Smale_diff}. Thus we will make a geometric outline of the proof. In fact for (i) we just notice that there is a Smale's horseshoe in neighborhood of $\{p,x\}$ since if we take a square  whose sides are transverse to $W^s(p,f)$ then the Theorem \ref{lambda_lemma} states that the  evolution theses sides by $f$ will be transform in a picture similar a horseshoe. It is well known that a horseshoe is topologically conjugated to the shift, for details see \cite[Chapter 3]{Wen}. For (ii) we  apply the Theorem \ref{lambda_lemma}  in $W^u(p,f)$ in order to obtain a $L$-homoclinic point $x\in W^u(p)\tpitchfork_{L,x}  W^s(p)$ and the result follows from (i). 
\end{proof}

Let $p$ and $q$ be $L$-hyperbolic fixed points of  a lipeomorphism $f$. We say that $p$ and $q$ are {\it $L$-heteroclinically related} if there is a point $x$ such that  $W^u(p,f)\tpitchfork_{L,x} W^s(q,f)$.

\begin{theorem}
The relation  $L$-heteroclinically related is transitive, that is, If $p,q,r$ are $L$-hyperbolic fixed points of a lipeomorphism $f$ such that $p$ and $q$ are  $L$-heteroclinically related, $q$ and $r$ are  $L$-heteroclinically related then $p$ and $r$ are  $L$-heteroclinically related.
\end{theorem}
\begin{proof}
Let $x\in W^u(p,f)\cap W^s(q,f)$ and $y\in W^u(q,f)\cap W^s(r,f)$ points in $M$ such that  $W^u(p,f)\tpitchfork_{L,x} W^s(q,f)$ and $W^u(q,f) \tpitchfork_{L,y} W^s(r,f)$.  We can choose two submanifolds $D_x$ and $D_y$  of $M$ such that $D_x\tpitchfork_{L,x}W^s (q,f)$ and $D_y \tpitchfork_{L,y}W^s (r,f)$ with $D_x\subset W^u(p,f)$ and $D_y\subset W^u(r,f)$ by Theorem \ref{lambda_lemma}  we find a intersection transverse point $z\in f^{n_x}(D_x)\cap  f^{-n_y}(D_y)$ and then $z\in W^u(q)\tpitchfork_{L,z}  W^s(r)$ for some positive integers $n_x$ and $n_y$. 
\end{proof}

\section{One dimensional dynamics}\label{One dimensional dynamics}
In this section we deal with one dimensional dynamics generated by a Lipschitz map. Since we have just one direction the dynamics have particular properties once that we can not have distinct direction of contraction and expansion. Thus the absence of saddle points makes hyperbolic fixed points  attracts or expands orbits in their neighborhood. In order to state our contributions we consider the following  preliminary examples. Let $f,g,h:\R\to \R$ maps given by
$f(x)=2x$ if $x<0$, $f(x)=x^2$ if $0.1\leq x\leq 1$ and  $f(x)= 0.5x+0.5$ if $x>1$, $g(x)=x^3-x$, for $x\in\R$ and $h(x)=0.2x$ if $<0$, $h(x)=x^2$ if $ 0.1\leq x\leq 1$ and $h(x)=2.1x-1.1$ if $x>1$.
%
We can see that $f, g$ and $h$ have two fixed points,  the map $g$ is smooth while $f$ and $h$ is not differentiable n a neighborhood of the fixed points. Thus the maps $f$ and $h$  does not fall into the general theory of smooth dynamical systems in order to study hyperbolicity and permanence of fixed points. Informally, if we move the graphs smoothly we see that the numbers of fixed points of $f$ can increase or disappear, so $f$ does not look well behaved under small perturbations. But this bad behavior is not due to lack of differentiability. In fact the map $h$ is not differentiable in the fixed points and even if we move its graphs smoothly we can see that the behavior of $h$ and $g$ are similar that is the fixed points are preserved.

\subsection{Permanence of Fixed Points}\label{sec2}
Our main goal in this subsection is to find a class of Lipschitz function whose the dynamics are preserved under small Lipschitz perturbations.   Our results states that not all class of Lipschitz maps will be stable under small perturbation in the Lip norm as we saw above the map $f$. We are concerned with maps like the above map $g$. In \cite{Giuliano} was introduced  the  concept of \emph{reverse Lipschitz map}. This concept was utilized in order to characterize sources for locally Lipschitz maps (with no differentiability). The map $f:{\mathbb R}\to{\mathbb R}$ is called {\it reverse Lipschitz} if there exists a constant $c \in {\mathbb R}$, $c > 0$ (reverse Lipschitz constant of $f$) such that, $\forall \ x, y \in {\mathbb R} 
\Longrightarrow | f(x) - f(y) | \geq c | x - y |$. Analogously, $f$ is said to be {\it locally reverse Lipschitz} if for all  $x\in\R$  there exists an $\varepsilon$-neighborhood
$N_{\epsilon} (x)$ of $x$ such that $f$ restricted to $N_{\epsilon}(x)$ is reverse Lipschitz. Notice that as local Lipschitz constant the reverse constant depend on the neighborhood of the considered point, that is, in general this constants may change from increasing or decreasing $\varepsilon$, thus we always consider the smallest constant in this neighborhood.  

We now define sink and source for Lipschitz maps without requiring differentiability.  Assume that $f:{\mathbb R}\to{\mathbb R}$ is a map 
and let $p$ be a fixed point of $f$. The point $p$ is a {\it sink} (or {\it attracting fixed point}) if there exists an $\epsilon > 0$
such that, for all $x \in N_{\epsilon}(p)$, $\displaystyle{\lim_{k\rightarrow\infty}}f^{k}(x) = p$. If all points sufficiently close to $p$ are repelled from $p$,
then $p$ is said to be a {\it source}. More precisely, $p$ is a source if there exists an epsilon neighborhood $N_{\epsilon} (p)$ such
that, for every $\ x \in N_{\epsilon} (p)$, $x \neq p$, there is a positive integer $k$ with $|f^{k}(x) - p| \geq \epsilon$. In \cite{Giuliano}, the authors  characterized sink and source for locally Lipschitz and reverse Lipschitz maps, respectively, by means  of the Lipschitz constant and reverse Lipschitz constant \cite[Theorem.3.2]{Giuliano}. Let $f:{\mathbb R}\to{\mathbb R}$ be a map and $p \in{\mathbb R}$ a fixed point of $f$. The following statements are true.
\begin{itemize}
\item [(i)] If $f$ is strictly locally Lipschitz map at $p$, with Lipschitz constant $c < 1$, then $p$ is a sink.
\item [(ii)] If $f$ is locally reverse Lipschitz map at $p$, with constant $r > 1$, then $p$ is a source.
\end{itemize}
The next results will improve this  by showing that sink and source are isolated fixed points which  are stable by small Lipschitz perturbation.

\begin{theorem}\label{main1}
Let $f:{\mathbb R}\to{\mathbb R}$ be a map and $p$ a fixed point of $f$. The following statements are true.
\begin{itemize}
\item [(i)] If $f$ is locally strictly Lipschitz, with constant $c < 1$ in a neighborhood of $p$, then $p$ is the unique fixed point in this neighborhood and it is a sink.
\item [(ii)] If $f$ is reverse Lipschitz with constant $r > 1$, in a neighborhood of $p$, then $p$ is the unique fixed point in this neighborhood and it is a source.
\end{itemize}
\end{theorem}
\begin{proof}
To prove (i) notice that follows  that $p$ is a sink, and then there is a neighborhood $N_\delta(p)$ such that $f(\overline{N_\delta(p)})\subset\overline{ N_\delta(p)}$. Now the result follows  from  Banach's Contraction Theorem.  In fact, let $q\in N_\delta(p)$ be a fixed point of $f$, with $q\neq p$. Take $\varepsilon=\frac{\delta-|p-q|}{2}$, then $N_\varepsilon(q)\subset N_\delta(p)$ and $q$ is a sink, thus, for $x\in N_\varepsilon(q)$, we have $\ds\lim_{k\to\infty}f^k(x)=q$ and $\ds\lim_{k\to\infty}f^k(x)=p$  which is a contradiction.  

To show  Item (ii) we have that $p$ is a source. Let $q\in N_\delta(p)$, $q\neq p $. Then there will be a positive integer $k_0$ such that $f^{k_0}(q)\notin N_\delta(p)$,  which implies $f(q)\neq q$. Therefore there is no  fixed point of $f$ different from $p$. 
\end{proof}

It is clear that if a map $f$ is such that Lip$(f)=1$ in a neighborhood of a fixed point $p$, then $f$ may exhibit different behaviors as we can see with the map $f(x)=x\cos(\ln(x))$, $x\neq 0$ and $f(0)=0$. We have $0$ a fixed point that has  attracting  and repelling direction.  This type of fixed point is called
{\it indiferent points}. Hence  sink and source are one dimensional version of the  $L$-hyperbolic points

\begin{lemma}\label{glips_too}
Let $f:\R\to\R$ be a map and let $p\in\R$. If $g:\R\to\R$ is a map such that $\|f-g \|_{N_{\delta} (p)}<\varepsilon$ (that is \eqref{Lip_Norm} is well defined in $N_{\delta} (p)$ and smaller than $\varepsilon$), for some $\delta>0$, then for $\varepsilon$ sufficiently small, we have:
\begin{itemize}
\item[(i)] If $f$ is strictly locally Lipschitz with locally Lipschitz constant  $c_{f,p}<1$ in  $N_\delta(p)$, then $g$ is locally Lipschitz and the locally Lipschitz constant of $g$ is strictly  less than one.
\item[(ii)] If $f$ is locally reverse Lipschitz with locally Lipschitz constant  $r_{f,p}>1$ in  $N_\delta(p)$, then $g$ is reverse locally Lipschitz and the  locally reverse Lipschitz constant of $g$ is strictly greater than one.
\end{itemize}
\end{lemma}
\begin{proof}
If $x,y\in N_\delta(p)$, $x\neq y$, 
\begin{align*}
|g(x)-g(y)|&\leq |g(x)-g(y)-f(x)+f(y)|+|f(x)-f(y)|\\
&\leq \varepsilon|x-y|+c_{f,p}|x-y|\leq (\varepsilon+c_{f,p})|x-y|.
\end{align*}
Since $c_{f,p}<1$, Item (1) follows if we take $\varepsilon<1-c_{f,p}$.

To prove Item (2), note that
$$
\varepsilon\geq \left|  \frac{f(x)-f(y)}{|x-y|}-\frac{g(x)-g(y)}{|x-y|} \right|\geq \left| \frac{f(x)-f(y)}{x-y}\right|-\left| \frac{g(x)-g(y)}{x-y} \right|, 
$$
which implies 
$$
\left| \frac{g(x)-g(y)}{x-y} \right|\geq r_{fp}-\varepsilon
$$
Since $r_{f,p}>1$, Item (2) follows if we take $\varepsilon$ sufficiently small.
\end{proof}

If we assume that $p$ is a $L$-hyperbolic fixed point of $f$, for example a sink, then  we conclude that $g$ is locally Lipschitz with locally Lipschitz constant strictly  less than one in $N_\delta(p)$.  But  this estimate can not be a priori transferred to all $\R$ and we also can not ensure that $g(N_\delta(p))\subset N_\delta(p)$.  Therefore we can not state yet that $g$ has a fixed point (sink) in $N_\delta(p)$.  As we will see in the next results we need to assume more smooth conditions in $f$ or impose restrictions in the Lipschitz constant of $f$. 

\begin{theorem}\label{main2}
Let $f:{\mathbb R}\to{\mathbb R}$ be a map and $p$ a fixed point of $f$ such that $f$ is sufficiently differentiable in $\R$ and  $|f'(p)|\neq 1$. If $g$ is a locally Lipschitz function such that $\|f-g \|_{N_{\delta} (p)}<\varepsilon$, then for $\delta$ and $\varepsilon$ sufficiently small there is a unique fixed point $q$ of $g$ in $N_\delta(p)$. Moreover, if $|f'(p)|<1$  then $q$ is a sink and if $|f'(p)|>1$ then $q$ is a source.
\end{theorem}

\begin{remark}
Note that we do not require any differentiability in the map $g$. Therefore Theorem \ref{main2} extends the permanence of equilibrium points known  when $f$ and $g$ are both continuously differentiable and the perturbation is with respect to the  $C^1$-norm.
\end{remark}
\begin{proof}
We start denoting $L=f'(p)$ and defining the auxiliary function $h(x)= g(x+p)-p$. Note that  $h(x-p)=g(x)-p$ and then $g(x)=x$ if only if $h(x-p)=x-p$. 
\begin{align*}
h(x-p)=x-p & \Leftrightarrow h(x-p)-L(x-p)=x-p-L(x-p)\\
&\Leftrightarrow (1-L)^{-1}[h(x-p)-L(x-p)]=x-p.
\end{align*}
If we denote $z=x-p$ and $\psi(z)=(1-L)^{-1}[h(z)-Lz]$, we have that $x$ is a fixed point for $g$ if only if $z$ is a fixed point of $\psi$. In what follows we prove that, for  $\varepsilon$ and $\delta$ sufficiently small,  $\psi$ is a strict contraction in  $N_\delta(p)$.  In fact, for $|x-p|\leq\delta$, we have
\begin{align*}
|\psi(x-p)|&\leq |(1-L)^{-1}[g(x)-p-L(x-p) ] |\\
& \leq|(1-L)^{-1}[g(x)-g(p)-f(x)+f(p)]|+|(1-L)^{-1}[g(p)-f(p)]|\\
&+|(1-L)^{-1}[ f(x)-f(p)-L(x-p)]|.
\end{align*}
By the Lipschitz closeness, we have
\begin{multline*}
|(1-L)^{-1}[g(x)-g(p)-f(x)+f(p)]|+|(1-L)^{-1}[g(p)-f(p)]|\\ \leq |1-L|^{-1}\varepsilon|x-p|+|1-L|^{-1}\varepsilon \leq |1-L|^{-1}\varepsilon\delta+|1-L|^{-1}\varepsilon.
\end{multline*}
Before that, since $f$ is differentiable we can take the remainder in the definition of differentiability such that
$$
|(1-L)^{-1}[ f(x)-f(p)-L(x-p)]|\leq |1-L|^{-1}\varepsilon|x-p|\leq |1-L|^{-1}\varepsilon\delta.
$$
Thus if we take $\varepsilon\leq\min\{|1-L|\delta/3,|1-L|/3 \}$ we have, $|\psi(z)|\leq \delta$.

Now we prove that $\psi$ is a contraction.
\begin{align*}
|\psi(z)-&\psi(\bar{z})|\leq |(1-L)^{-1}[g(x)-p-L(x-p)-g(\bar{x})+p+L(\bar{x}-p) ]| \\
&= |(1-L)^{-1}[g(x)-g(\bar{x})-L(x-\bar{x})| \\
&\leq|(1-L)^{-1}[g(x)-g(\bar{x})-f(x)+f(\bar{x})]|+|(1-L)^{-1}[f(x)-f(\bar{x})-L(x-\bar{x})]|. 
\end{align*}
If we proceed as above, we obtain
$$
|(1-L)^{-1}[g(x)-g(\bar{x})-f(x)+f(\bar{x})]|\leq |1-L|^{-1}\varepsilon |x-\bar{x}|
$$
and since $f$ is sufficiently differentiable, for $\delta$ sufficiently small,
\begin{align*}
|(1-L)^{-1}[f(x)-f(\bar{x})-L(x-\bar{x})]|&\leq |(1-L)^{-1}[f(x)-f(\bar{x})-f'(\bar{x})(x-\bar{x})]|\\
&+|(1-L)^{-1}[(f'(\bar{x})-f'(p))(x-\bar{x})]|\\
& \leq |(1-L)|^{-1}\varepsilon|x-\bar{x}|+|(1-L)|^{-1}\varepsilon|x-\bar{x}|.
\end{align*}
Thus, if we take $\varepsilon< |1-L|/6$ we obtain
$$
|\psi(z)-\psi(\bar{z})|\leq 3|1-L|^{-1}\varepsilon |x-\bar{x}|< \tfrac{1}{2}|x-\bar{x}|=\tfrac{1}{2}|z-\bar{z}|,
$$
that is,
$
|\psi(z)-\psi(\bar{z})|< \tfrac{1}{2}|z-\bar{z}|.
$
Hence there is a unique $q\in \overline{N_\delta(p)}$ such that $g(q)=q$.

Finally, since $f$ is continuously differentiable, then $f$ is locally Lipschitz. If we assume that $|f'(p)|<1$ then we can take $\delta,\lambda>0$ such that $|f'(x)|<\lambda<1$ for all $x\in N_\delta(p)$, by Mean Value Theorem the locally Lipschitz constant of $f$ in $N_\delta(p)$ is $\sup_{x\in N_\delta(p)}|f'(x)|<1$, Now the result follows from  Lemma \ref{glips_too} and Theorem \ref{main1}. Analogously we obtain that if $|f'(p)|>1$ then for $\delta$ sufficiently small $q$ is a sink for $g$.
\end{proof}

Now we drop the differentiability of $f$ and  we exhibit a class of locally Lipschitz maps whose it is possible  ensure the permanence and stability of fixed points. This class involves maps like $h$ and excludes  maps like $f$ presented in introduction of this section.

\begin{theorem}\label{main3}
Let $f:{\mathbb R}\to{\mathbb R}$ be a map and $p$ a fixed point of $f$ such that $f$ is locally Lipschitz with constant $C\neq 1$ in the neigborhood $N_\delta(p)$, for some $\delta>0$. Assume that $f$ satisfies
\begin{equation}\label{gordura}
|(1-C)^{-1}[f(x)-f(y)-C(x-y)|\leq \tfrac{1}{3}|x-y|,\quad\tn{for all }x,y\in N_\delta(p).
\end{equation}
If $g$ is a locally Lipschitz map such that $\|f-g \|_{N_{\delta} (p)}<\varepsilon$, then for $\varepsilon$ sufficiently small,  there is a unique fixed point $q$ of $g$ in $N_\delta(p)$. Moreover if $C<1$ then $q$ is a sink for $g$ and if $f$ is reverse Lipschitz with constant $r>1$ then $q$ is a source for $g$. 
\end{theorem}
\begin{proof}
With the same argument of Theorem \ref{main2}, we obtain,
\begin{align*}
|\psi(x-p)| & \leq|(1-C)^{-1}[g(x)-g(p)-f(x)+f(p)]|+|(1-C)^{-1}[g(p)-f(p)]|\\
&+|(1-C)^{-1}[ f(x)-f(p)-C(x-p)]|.
\end{align*}
We take $\epsilon<\min\{|1-C|\delta/3,|1-C|/3 \}$ it follows from \eqref{gordura} that $\psi$ takes $\overline{N_\delta(p)}$ into itself. Again as in the Theorem \ref{main2}, we have for $z=x-p$ and $\bar{z}=\bar{x}-p$,
$$
|\psi(z)-\psi(\bar{z})| \leq|(1-C)^{-1}[g(x)-g(\bar{x})-f(x)+f(\bar{x})]|+|(1-C)^{-1}[f(x)-f(\bar{x})-L(x-\bar{x})]|. 
$$
If we take $\epsilon< |1-C|/2$ it follows from \eqref{gordura} that $\psi$ is a strict contraction in $\overline{N_\delta(p)}$, hence $g$ has a unique fixed point $q$ in $\overline{N_\delta(p)}$. The stability of $q$ it follows from Lemma \ref{glips_too} and Theorem \ref{main1}.
\end{proof}

Notice that $f$ is sufficiently differentiable then the condition \eqref{gordura} is always true in a neighborhood of a $p=f(p)$. Moreover if $f$ is locally Lipschitz  we can rewrite \eqref{gordura} in the form
$$
\left|\frac{f(x)-f(y)}{x-y}-C \right|\leq \frac{|1-C|}{3},\quad x\neq y.
$$ 
This means that for functions like $f$ and $h$ presented in the introduction of this section, the locally Lipschitz constants in $N_\delta(p)$ should be closed of $\frac{df^+(p)}{dx}$ and $\frac{df^-(p)}{dx}$. Now it is easy to see that $h$ and $g$ satisfy \eqref{gordura} whereas $f$ not satisfies this condition.

\subsection{Permanence of Periodic Points}\label{sec3}
In this section we are concerned with permanence of periodic points. The differentiable characterization of periodic points can be found in \cite{Chaos}. But if $g=f^{k}:{\mathbb R}\to{\mathbb R}$ be a map and $p\in {\mathbb R}$ a fixed point of $g$, the following statements are true see \cite{Giuliano}.
\begin{itemize}
\item [(i)] If $g$ is strictly locally Lipschitz map at $p$, with Lipschitz constant $c < 1$, then $p$ is a periodic sink.
\item [(ii)] If $g$ is locally reverse Lipschitz map at $p$, with constant $r > 1$, then $p$ is a periodic source.
\end{itemize}
Since a $k$-periodic point is a fixed point for the map $f^k$ we can apply the Theorems \ref{main2} and \ref{main3} to obtain the following results. 
\begin{theorem}\label{main2_periodic}
Let $f:{\mathbb R}\to{\mathbb R}$ be a map and $p$ a fixed point of $f^k$, $k>1$, such that $f$ is continuously differentiable in $\R$ and  $|(f^k)'(p)|\neq 1$. If $g$ is a locally Lipschitz function such that $\|f^k-g^k \|_{N_{\delta} (p)}<\varepsilon$, then for $\delta$ and $\varepsilon$ sufficiently small there is a unique fixed point $q$ of $g^k$ in $N_\delta(p)$, that will be a preperiodic point of $g$. Moreover, if $|(f^k)'(p)|<1$  then the orbit of $q$ y $g$ is a periodic sink and if $|(f^k)'(p)|>1$ then the orbit of $q$ by $g$ is a periodic source.
\end{theorem}
The Lipschitz version of the Theorem \ref{main2_periodic} without requiring differentiability in $f$ can be state as follows.
\begin{theorem}\label{main3_periodic}
Let $f:{\mathbb R}\to{\mathbb R}$ be a map and $p$ a fixed point of $f^k$, $k>1$, such that $f^k$ locally Lipschitz in $\R$ with locally Lipschitz constant $C\neq 1$ in the neigborhood $N_\delta(p)$, for some $\delta>0$. Assume that $f^k$ satisfies
\begin{equation}\label{gordura2}
|(1-C)^{-1}[f^k(x)-f^k(y)-C(x-y)|\leq \tfrac{1}{3}|x-y|,\quad\tn{for all }x,y\in N_\delta(p).
\end{equation}
\par If $g$ is a locally Lipschitz function such that $\|f^k-g^k \|_{N_{\delta} (p)}<\varepsilon$, then for $\varepsilon$ sufficiently small there is a unique fixed  point $q$ of $g^k$ in $N_\delta(p)$. Moreover, if $C<1$  then the orbit of $q$ by $g$  is a periodic sink and if $C>1$ then the orbit of $q$ by $g$ is a periodic source.
\end{theorem}
Now consider a small Lipschitz perturbation of the following Lipschitz Logistic map. Let $g(x)=3.3x(1-x)$ be a map,  we have (see \cite{Chaos}) that $g$ has $2$-periodic sink orbit $\{0.4794,08236\}$  (considering four decimal places accuracy). We now can break the differentiability of $g$ into $0.4794$ and make a small perturbation. That is, we define the locally Lipschitz maps  
\begin{equation*}
f(x)=
\begin{cases}
0.15x+0.75,\, x<0.4794,\\
3.3x(1-x),\, x\geq 0.4794, 
\end{cases}\tn{ and }\,\,
h_\varepsilon(x)=
\begin{cases}
0.15x+0.75+\varepsilon,\, x<0.4794,\\
3.3x(1-x)+\varepsilon,\,  x\geq 0.4794. 
\end{cases}
\end{equation*}
Note that $f$  is not differentiable in $0.4794$ but   $\{0.4794,08236\}$ is also a periodic orbit for$f$ and the product of the the locally Lipschitz constant in a neigborhood of $0.4794$ and $0.8236$  is less than one, hence  it is a periodic sink. Moreover $h_\varepsilon$ is a Lipschitz perturbation of $h$ and then, by Theorem \ref{main3_periodic} for $\varepsilon$ sufficiently small $h_\varepsilon$ also has a periodic orbit  with same stability of $h$. 

\section{Lyapunov Exponent}\label{Lyapunov Exponent}
We finish our work with a discussion about Lyapunov exponent. We only consider the case of maps defined over $\R$, since the procedure for maps on $\R^n$ is quite similar. We propose the folowing definition of Lyapunov number and Lyapunov exponent for Lipschitz maps, which are not necessarily differentiable. 

\begin{definition}\label{L1}
Let $f:{\mathbb R}\to{\mathbb R}$ be a locally Lipschitz map and $\delta>0$. Denote $C_{x_i,\delta}$ the locally Lipschitz constant of $f$ in $N_\delta(x_i)$, $i=1,2,...$. Then the $\delta$-Lyapunov number $L_\delta(x_1)$ of the orbit ${\mathcal O}_{x_1}= \{ x_1 , x_2, x_3 ,\ldots \}$ is defined as
\begin{equation}\label{Lyap1}
L_\delta(x_1)=\displaystyle{\lim_{n\rightarrow\infty}}(C_{x_1,\delta}\cdots
C_{x_n,\delta})^{1/n},
\end{equation}
if the limit exists.

The $\delta$-Lyapunov exponent $h_\delta(x_1)$ is defined as
\begin{equation}\label{Lyap2}
h_\delta(x_1)=
\displaystyle{\lim_{n\rightarrow\infty}}(1/n)[\ln C_{x_1,\delta}+\cdots
+\ln C_{x_n,\delta}],
\end{equation}
if the limit exists.
\end{definition}

Note that since $f$ is locally Lipschitz \eqref{Lyap1} and \eqref{Lyap2} are well defined for $\delta>0$. Moreover, if $f$ is continuously differentiable with nonzero derivative, then for all $x_1\in\R$, we have
$
\lim_{\delta\to 0} L_\delta(x_1)=L(x_1).
$
Let $f: \mathbb{R}\to \mathbb{R}$ be a map. We say that $f$ is {\it locally $\delta$-uniform Lipschitz} map (where $\delta > 0$ is a real number) if, for every point $x$ of the domain, there exists a $\delta$ neighborhood $N_{\delta}(x)$ of $x$ such that $f$ restricted to $N_{\delta}(x)$ is Lipschitz. In other words, for a fixed $\delta$, every point has a $\delta$ neighborhood such that $f$ restricted to it is Lipschitz. The following result is a variation of \cite[Theorem 3.4]{Chaos} based on Lipschitz maps.

\begin{theorem}\label{L4}
Let $f: {\mathbb R}\to {\mathbb R}$ be a $\delta$-uniform  locally strictly Lipschitz map for some $\delta>0$.  Assume that ${\mathcal O}_{x_1}=\{x_1,x_2,...\}$ is asymptotically periodic to the periodic orbit ${\mathcal O}_{y_1} = \{y_1, y_2, \ldots\}$, then there is $\bar{\delta}>0$ such that $h_{\bar{\delta}}(x_1)\leq h_\delta(y_1)$, if both Lyapunov exponent exist.
\end{theorem}
\begin{proof}
Assume ${\mathcal O}_{y_1} = \{y_1, y_2, \ldots\}=\{y\}$, that is, $y$ is a fixed point of $f$. Then $x_n\to y$ as $n\to\infty$ and
$$
\lim_{n\to\infty}\left| \frac{f(x_n)-f(z)}{x_n-z}\right|=\left| \frac{f(y)-f(z)}{y-z}\right|< C_{y,\delta},\quad z\neq x_n,z\neq y, z\in N_\delta(y),
$$
where $C_{y,\delta}$ denotes the locally Lipschitz constant of $f$ in $N_\delta(y)$.  Thus, for $n_\delta$ sufficiently large, there is $\bar{\delta}>0$ such that $N_{\bar{\delta}}(x_k)\subset N_\delta(y)$ for all $k>n_\delta$ and 
$$
\left| \frac{f(x_k)-f(z)}{x_k-z}\right|< C_{y,\delta},\quad z\neq x_k, z\in N_{\bar{\delta}}(x_k),
$$
which implies $C_{x_k,\bar{\delta}}\leq C_{y,\delta}$ for all $k>n_\delta$, hence $h_{\bar{\delta}}(x_1)\leq h_\delta(y)$. 

If $k > 1$, we know that $y_1$ is a fixed point of $f^{k}$ (which is also locally Lipschitz)  and ${\mathcal O}_{x_1}$ is asymptotically periodic under $f^{k}$ to ${\mathcal O}_{y_1}$. Applying the reasoning above to $x_1$ and $f^{k}$ it follows that $h^k_{\bar{\delta}}(x_1)\leq h^k_\delta(y_1)$. The result follows noting that $h_{\bar{\delta}}(x_1)=\frac{1}{k}h_{\bar{\delta}}^k(x_1)$.
\end{proof}

Consider the locally Lipschitz maps  
\begin{equation*}
f(x)=
\begin{cases}
0.15x+0.75,\, x<0.4794,\\
3.3x(1-x),\, x\geq 0.4794, 
\end{cases}\tn{ and }\,\,
\end{equation*}
We calculate the $\delta$-Lyapunov exponent of the periodic orbit $\{y_1,y_2\}=\{0.4794,08236\}$. It is easy to see that for each $\delta>0$,
\begin{equation}\label{expliaplip}
C_{y_1,\delta} C_{y_2,\delta}\leq 0.15\cdot 2.1357\cdot 2\delta
\end{equation}
Since this periodic orbit is a sink, every orbit ${\mathcal O}_{x_1}=\{x_1,x_2,...\}$ asymptotically to $\{x_1,x_2\}$ have $\delta$-Lyapunov exponent bounded by \eqref{expliaplip}. 

\bibliographystyle{abbrv}
\bibliography{References}

\begin{thebibliography}{10}

\bibitem{Giovanni_2010}
G.~Alberti, M.~Csornyei, and D.~Preiss.
\newblock Differentibility of lipschitz functions, structure of null sets, and
  other problems.
\newblock {\em Proceedings to the International Congress of Matemathicians,
  Hyderabad, India}, 2010.

\bibitem{Chaos}
K.~T. Alligood, T.~D. Sauer, and J.~A. Yorke.
\newblock {\em CHAOS: An Introduction to Dynamical Systems}.
\newblock Springer, New york, 1996.

\bibitem{GeraldBeer}
G.~Beer and M.~J. Hoffman.
\newblock The lipschitz metric for real-valued continuous functions.
\newblock {\em Journal of Math. Anal. Applications}, 406:229--236, 2013.

\bibitem{lip_banach}
M.~C. Bortolan, C.~Cardoso, A.~N. Carvalho, and L.Pires.
\newblock Lipschitz perturbations of morse-smale semigroups.
\newblock {\em Preprint arXiv:1705.09947}, 2019.

\bibitem{Craig_2008}
C.~Calcaterra and A.~Boldt.
\newblock Lipschitz flow-box theorem.
\newblock {\em Journal of Math. Anal. Applications}, 338:1108--1115, 2008.

\bibitem{Cobzas2019}
{\c S}.~Cobza{\c s}, R.~Miculescu, and A.~Nicolae.
\newblock Lipschitz functions.
\newblock Springer Nature Switzerland, 2019.
\newblock Lecture Notes in Mathematics, 2241.

\bibitem{Evansc}
L.~Evans.
\newblock Measure theory and fine properties of functions.
\newblock CRCPress,Boca Raton, 1992.

\bibitem{Garrido_2013}
M.~Garrido, J.A.Jaramillo, and Y.C.Rangel.
\newblock Smooth aproximation of lipschitz functions on finlser manifold.
\newblock {\em Journal of Function Spaces and Applications}, pages 1--10, 2013.

\bibitem{Giuliano}
G.~G.~L. Guardia and P.~J. Miranda.
\newblock Lyapunov exponent for lipschitz maps.
\newblock {\em Nonlinear Dynamics}, 92(3):1217--1224, 2018.

\bibitem{Juha_2004}
J.~Heinonen.
\newblock {\em Lectures on Lipschitz analysis}.
\newblock 14thJyvaskyla Summer Scholl, 2004.

\bibitem{hpg}
M.~Hirsch, C.~Pugh, and M.~Shub.
\newblock {\em Invariant Manifold, Lectures note in math}, volume 583.
\newblock Springer-Verlag New York, 1977.

\bibitem{katok}
A.~Katok and B.~Hasselblatt.
\newblock {\em Introduction to the modern theory of smooth dynamical systems}.
\newblock Cambridge University Press, 1995.

\bibitem{Movahedi-Lankarani2003}
H.~Movahedi-Lankarani and R.~Wells.
\newblock $c^1$-weierstrass for compact sets in hilbert space.
\newblock {\em Journal of Math. Anal. Applications}, 285:299--320, 2003.

\bibitem{Palis1968}
J.~Palis.
\newblock On morse-smale dynamical systems.
\newblock {\em Topology 8}, pages 395--404, 1968.

\bibitem{Palis1982}
J.~Palis and W.~de~Melo.
\newblock {\em Geometric Theory of Dynamical Systems. An introduction}.
\newblock Springer-Verlag New York, 1982.

\bibitem{Palis1070}
J.~Palis and S.~Smale.
\newblock Structural stability theorems.
\newblock {\em poc.symp. pure math. amer. Math. Soc.}, 14:223--232, 1970.

\bibitem{Peixoto}
M.~Peixoto.
\newblock Structural stability on two dimensional manifolds.
\newblock {\em Topology I}, pages 101--120, 1962.

\bibitem{pilyudin1994}
S.~Y. Pilyugun.
\newblock {\em The space of dynamical systems with the $C^0$-topology}.
\newblock Springer, 1994.

\bibitem{Robinson_chaos}
C.~Robinson.
\newblock {\em Dynamical systems, stability, symbolic dynamics and chaos}.
\newblock CRC Press, 1995.

\bibitem{shub}
M.~Schub.
\newblock {\em Global Stability of Dynamical Systems}.
\newblock Springer-Verlag New York, 1987.

\bibitem{D.Sherbert}
D.~Sherbert.
\newblock Banach algebras of lipschitz functions.
\newblock {\em Pacific J. Math}, 13:1387--399, 1963.

\bibitem{SMALEg}
S.~Smale.
\newblock Diffeomorphisms with many periodic points.
\newblock Differential and combinatorial Topology, 1965.

\bibitem{Smale_diff}
S.~Smale.
\newblock Differentiable dynamical systems.
\newblock {\em Bull. Amer.Math. Soc.}, 73:747--817, 1967.

\bibitem{Kentaro_1974}
K.~Takaki.
\newblock Lipeomorphisms close to an anosov diffeomorphism.
\newblock {\em Nagoya MathJ.}, 53:71--82, 1974.

\bibitem{Wen}
L.~Wen.
\newblock {\em Differentiable Dynamical Systems: An Introduction to Structural
  Stability and Hyperbolicity}.
\newblock Peking University, Beijing, China, 2016.

\end{thebibliography}
\end{document}